\documentclass[manuscript,screen,sigconf,nonacm]{acmart}

\AtBeginDocument{%
	\providecommand\BibTeX{{%
			\normalfont B\kern-0.5em{\scshape i\kern-0.25em b}\kern-0.8em\TeX}}}

\setcopyright{acmcopyright}
\copyrightyear{2023}
\acmYear{2023}
\acmDOI{XXXXXXX.XXXXXXX}

\acmConference[ISSAC '23]{International Symposium on Symbolic and Algebraic Computation 2023}{July 24-27 2023}{Troms\o, Norway}
\acmPrice{15.00}
\acmISBN{978-1-4503-XXXX-X/18/06}

\usepackage{quasimorphisms}  
\usepackage{commands}

\usepackage{enumitem}
\usepackage{mathtools}
\usepackage{color}
\usepackage{graphicx}
\graphicspath{ {images/} }
\usepackage{transparent}
\usepackage{float}
\usepackage[font=small]{caption}
\usepackage{lipsum}
\usepackage{a4wide}
\usepackage{cleveref}
\usepackage[shortcuts]{extdash}

\usepackage{color}
\newcommand\blue[1]{{\color{blue}#1}}
\newcommand\red[1]{{\color{red}#1}}
\newcommand{\Norm}[1]{\left\lVert#1\right\rVert}
\newcommand{\palinsum}[1]{{B_{#1}}}

\newcommand{\altDelta}{\tilde{\Delta}}

\newcommand{\class}[2]{[#1]_{#2}}

\newtheoremstyle{case}{}{}{}{}{}{:}{ }{}
\theoremstyle{case}

\begin{document}  		
	
	\title{One-quasihomomorphisms from the integers into symmetric matrices}

	\author{Tim Seynnaeve}
	\affiliation{%
		\institution{KU Leuven}
		\country{Belgium}
	}
	\email{tim.seynnaeve@kuleuven.be}
	
	\author{Nafie Tairi}
	\affiliation{%
		\institution{Universit{\"a}t Bern}
		\country{Switzerland}}
	\email{nafie.tairi@unibe.ch}
	
	\author{Alejandro Vargas}
	\affiliation{%
		\institution{Nantes Université}
		\country{France}
	}
	\email{alejandro@vargas.page}
	
	\renewcommand{\shortauthors}{Tim Seynnaeve, Nafie Tairi, Alejandro Vargas}
	
	\begin{abstract}
		A function $f$ from $\ZZ$ to the symmetric matrices over an arbitrary field $K$ of characteristic 0 is a 1\-/quasihomomorphism if the matrix $f(x+y) - f(x)-f(y) $ has rank at most 1 for all $x,y \in \ZZ$.
		We show that any such $1$\-/quasihomomorphism has distance at most 2 from an actual group homomorphism. This gives a positive answer to a special case of a problem posed by Kazhdan and Ziegler.
	\end{abstract}


	\maketitle
	
	\section{Introduction}
	
	We continue the program initiated in \cite{DESTV22} of studying particular instances of a problem posed by Kazhdan and Ziegler in their work on approximate cohomology \cite{KZ18}. 
	We are given a function $f$ that behaves roughly like a homomorphism, in the following manner.

	\begin{de}
		Let $(H,+)$ be an abelian group.
		A \emph{norm} on $H$ is a map $\norm{\cdot}: H \to \mathbb{R}$ such that 
		\begin{itemize}
			\item $\norm{x} \geq 0$ for all $x \in H$, with equality if and only if $x=0$,
			\item $\norm{x+y} \leq \norm{x} + \norm{y}$ for all $x,y \in H$,
			\item $\norm{-x} = \norm{x}$ for all $x \in H$. 
		\end{itemize}
		Note that equipping $H$ with a norm is equivalent to equipping it with an equivariant metric $d$, that is, a metric such that $d(x,y)=d(x+z,y+z)$ for all $x,y,z \in H$; the connection is given by $d(x,y)=\norm{x-y}$.
	\end{de}

	\begin{de} 
		\label{de:cQuasiMorphism}
		Let $(G,+)$ and $(H,+)$ be abelian groups, where $H$ is equipped with a norm $\norm{\cdot}$.
		A map $f : G \to  H $ is a \emph{$c$\-/quasihomomorphism} (where $c \in \mathbb{R}_{\geq 0}$)
		if for all $x, y \in \ZZ$ we have that 
		\begin{equation} \label{eq:quasiMorphismGeneral}
			\norm{f(x+y) - f(x) - f(y)} \le c.
		\end{equation}
	\end{de}
	
	The natural question is whether every $c$\-/quasihomomorphism can be approximated by an actual group homomorphism. 
	\begin{que} \label{que:main}
		Fix $G$, $H$ and $c$. Does there exist a constant $C \in \mathbb{R}_{\geq 0}$ such that for every $c$\-/quasihomomorphism $f: G \to H$, there exists a group homomorphism $\varphi: G \to H$ such that
		\[
		\forall x \in G: \quad \norm{f(x) - \varphi(x)} \leq C.
		\]
	\end{que}
	
	A variant of this question, where $G=H$ and $G$ can be nonabelian, was asked already by Ulam \cite[Chapter VI.1]{Ulam} in 1960. 
	Our case of interest is when $G=\ZZ$ is the additive group of integers, and $H$ is the additive group of matrices over some field $\ourfield$, where the norm is given by the rank. The argument from \cite[Remark 1.11]{DESTV22} shows that in this case the answer is affirmative for fields of positive characteristic. For the rest of the paper, we will fix a field $\ourfield$ of characteristic $0$.
	Note that every group morphism $\varphi: \ZZ \to H$ is of the form $\varphi(x)=x\cdot A$, where $A\in H$ is a fixed element. 
	\begin{que} \label{que:mainSpecial}
		Fix $c \in \NN$. Does there exist a constant $C \in \mathbb{R}_{\geq 0}$ such that for every natural number $n$ and every $c$\-/quasihomomorphism $f: \ZZ \to \operatorname{Mat}(n\times n,\ourfield)$, there exists a matrix $A \in \operatorname{Mat}(n\times n,\ourfield)$ such that
		\[
		\forall x \in \ZZ: \quad \rk{f(x) - x\cdot A} \leq C.
		\]
	\end{que}
	
	This is the instance of Question \ref{que:main} asked by Kazhdan and Ziegler.
	It was answered affirmatively in \cite{DESTV22} under the assumption that $f$ lands in the space of diagonal matrices and by choosing $C=28c$. 
	In this paper, we study the case $c=1$. 
	We are able to prove a much better bound than the predicted $C=28$: indeed, the constant $C$ can be chosen equal to $2$. 
	Moreover, we can weaken the assumption that $f$ lands in the space of diagonal matrices.

	\begin{thm} 
		\label{thm:GeneralCase}
		Assume $\operatorname{char}(\ourfield)=0$ and let $\Sym{n \times n, \ourfield}$ be the space of symmetric matrices.
		If $f: \ZZ \to \Sym{n \times n, \ourfield}$ is assumed to be a 1\-/quasihomomorphism,
		there is an $A \in \Sym{n \times n, \ourfield} $ such~that 
		\begin{equation} \label{eq:mainTheorem}
			\rk{f(x) - x\cdot A} \le 2 \quad \forall \, x \in \ZZ.
		\end{equation}  
	\end{thm}
	
	The rest of this paper is devoted to proving Theorem \ref{thm:GeneralCase}. 
	The strategy is to prove that the sequence of consecutive differences $\Delta_f(i) = f(i+1) - f(i)$ satisfies two kinds of symmetries.
	One is a reflection symmetry in a local sense, which we call palindromicity.
	The other is a periodicity. 
	By expressing $f$ as a sum of deltas, and applying the symmetries, we arrive to the result.

	\section{Lemmas about symmetric matrices}

	In this section we prove some elementary lemmas about symmetric matrices that we will use later during the proof. %
	Let $(\cdot, \cdot)$ be the bilinear form on $\KK^n$ given by 
	\[
	(x,y) \coloneqq x_1y_1 + \cdots + x_ny_n
	\]
	for all $x,y \in \KK^n$. Then an $n \times n$ matrix over $\KK$ is symmetric if that $(Ax,y) = (x,Ay)$ for all $x,y \in \KK^n$. 
	
	\begin{lm} \label{lm:symperp}
		Let $A \in \Sym{n \times n, \KK}$. Then $\im(A) = \ker(A)^{\perp}$.
	\end{lm}
	\begin{proof}
		Since our bilinear form is nondegenerate we see that
		\[
		Ax = 0  \iff  (Ax,y)=0 \ \ \forall y  \iff  (x,Ay) = 0 \ \ \forall y.
		\]
		Therefore, 
		\[
		x \in \ker(A)  \iff x \perp \im(A),
		\]
		which means that $\im(A) = \ker(A)^{\perp}$.
	\end{proof}
	
	\begin{lm} \label{le:symmatmain}
		Let $A,B$ be symmetric matrices. Moreover, suppose that $\im (A) \cap \im (B) = 0$. 
		Then $\rk{A+B} = \rk{A}+\rk{B}$.
	\end{lm}
	\begin{proof}
		We always have inequalities
		\begin{align*}
			\rk{A+B} &= \dim \im (A+B) \\
			&\leq \dim (\im(A) + \im(B)) \\
			&\leq \dim(\im (A)) + \dim(\im(B))\\
			&= \rk{A} + \rk{B}.
		\end{align*}
		Our assumption $\im(A) \cap \im(B) = 0$ implies that the second ``$\leq$" is an equality. We show that the first ``$\leq$" is an equality as well. For this we need to show that $\im(A+B)=\im(A)+\im(B)$. Taking $\perp$ of both sides and applying Lemma \ref{lm:symperp}, this is equivalent to showing $\ker(A+B)=\ker(A)\cap \ker(B)$. But this again follows from our assumption $\im(A) \cap \im(B) = 0$:
		\begin{gather*}
			v \in \ker(A+B) \implies Av=-Bv \implies \\
			Av=Bv=0 \implies v \in \ker(A) \cap \ker(B). \qedhere
		\end{gather*}
	\end{proof}
	In fact, we will only need the following corollaries:
	\begin{cor} 
		\label{cor:ABmatrices}
		Let $A, B$ be symmetric matrices.
		If $\rk B = 1$ and $\symmspan B \not \subset \symmspan A$, 
		then  $\rk{A+B} = \rk{A} + 1$.
	\end{cor}
	\begin{proof}
		This is just the main claim for $B$ of rank one. %
	\end{proof}
	\begin{cor} \label{cor:symMatrixRank2}
		Let $A \in \Sym{n \times n, \ourfield}$ with $\rk{A} \leq 2$. Assume there are three rank-1 symmetric matrices $B_i$ ($i=1,2,3$) such that $\dim(\symmspan{B_1} + \symmspan{B_2} + \symmspan{B_3} ) = 3$ and $\rk{A-B_i} \le 1$ for $i=1,2,3$. Then $A=0$.
	\end{cor}
	\begin{proof}
		Suppose by contradiction that $\rk{A} \geq 1$. Then $$
		\rk{A-B_i}\leq 1 < 2 \leq \rk{A}+\rk{B_i},
		$$ thus by the contraposition of Lemma \ref{le:symmatmain}, 
		it follows that $\im(B_i) \subseteq \im(A)$. However, this would imply that 
		\[
		3=\dim(\im(B_1)+\im(B_2)+\im(B_3)) \leq \dim(\im(A)) = \rk{A} \leq 2,
		\]
		which is a contradiction.
	\end{proof}
	
	\section{Delta sequence}
	
	We begin by arguing that without loss of generality, we can assume that $f(1)=0$. This follows from the following observation.
	\begin{observation} \label{lemma:f(1)=0} 
		Let $H$ be a normed abelian group and $f:\ZZ \to H$ any function. If $g$ is defined by 
		\[g(x)=f(x)+ x\cdot C,\] 
		where $C \in  H $, then: %
		\begin{itemize}
			\item $f$ is a 1\-/quasihomomorphism if and only if $g$ is.
			\item We have that
			\[
			\norm{f(x) - x\cdot A} \le 2 \iff \norm{g(x) - x\cdot A'} \le 2,
			\]  
			where $A'=A-C$
		\end{itemize}
		Hence, by choosing $C=-f(1)$, we see that proving Theorem \ref{thm:GeneralCase} under the additional assumption $f(1)=0$ is enough to prove it in general.
	\end{observation}
	
	From now on we always assume $f(1)=0$. This allows us to reformulate the condition of $f$ being a 1\-/quasihomomorphism in terms of a difference operator on~$f$.
	
	\begin{de} 
		\label{de:DeltaSequence}
		Given a function $f : \ZZ \to H$, 
		we define its \emph{delta map} $\Delta_f(x) : \ZZ \to H$ as 
		\[ \Delta_f(x) = f(x+1) - f(x).  \] 
	\end{de}
	
	\begin{re}
		If $f(1)=0$, we can write $f$ in terms of $\Delta_f$:
		\begin{equation}\label{eq:findeltapos}
			f(x)=\sum_{i=1}^{x-1}{\Delta_f(i)} \text{ for } x \geq 1,
		\end{equation}
		and
		\begin{equation}\label{eq:findeltaneg}
			f(x)=-\sum_{i=0}^{x}{\Delta_f(i)} \text{ for } x \leq 0.
		\end{equation}
	\end{re}

	\begin{lm} 
		\label{lm:cQuasiMorIFFrkDeltaSeqlessc}
		Let $f : \ZZ \to H$ be a map with $f(1)=0$.
		The map $f$ is a $c$\-/quasihomomorphism 
		if and only if for all $k \in \ZZ_{\ge 0}$ and $z \in \ZZ$ we have %
		
		\begin{align}
			\label{eq:PositiveSide}
			\Norm{\sum_{i=1}^{k} \Delta_f(i) - \sum_{i=0}^{k} \Delta_f(z-i)}  &\le c,  \\
			\Norm{\sum_{i=0}^{k} \Delta_f(-i) - \sum_{i=0}^{k-1} \Delta_f(z-i)}  &\le c. \label{eq:NegativeSide}
		\end{align} 
		
	\end{lm}
	
	\begin{proof}
		In essence, this is just plugging in Equations~\eqref{eq:findeltapos} and \eqref{eq:findeltaneg} into Equation~\eqref{eq:quasiMorphismGeneral}. We present the proof in a slightly different way, to avoid doing case distinctions on the signs of $x$, $y$, and $x+y$. Calculate: 
		\begin{align*}
			& \sum_{i=1}^{k} \Delta(i) - \sum_{i=0}^{k} \Delta(z-i)  = \\
			& \sum_{i=1}^{k} [f(i+1) - f(i)]
			- \sum_{i=0}^{k} [f(z-i+1) - f(z-i)] = \\
			& f(k+1) +  f(z-k) - f(z+1).
		\end{align*} 
		By setting $x=k+1, y = z-k$, 
		we see that Equation~\eqref{eq:PositiveSide} holds if and only if 
		the $c$\-/quasihomomorphism condition~\eqref{eq:quasiMorphismGeneral} is fullfilled for $x \in \ZZ_{\ge 1}$ and $y \in \ZZ$.
		Similarly, calculate: 
		\begin{align*}
			& \sum_{i=0}^{k} \Delta(-i) - \sum_{i=0}^{k-1} \Delta(z-i) = \\
			& \sum_{i=0}^{k} [f(-i+1) - f(-i)] 
			- \sum_{i=0}^{k-1} [f(z-i+1) - f(z-i)] = \\
			& -f(-k) - f(z+1) + f(z+1-k).
		\end{align*} 
		By setting $x=-k, y = z+1 $, %
		we see that Equation~\eqref{eq:NegativeSide} is equivalent to the quasihomomorphism condition for $x \in \ZZ_{\le 0}$ and $y \in \ZZ$, and we are done. 
	\end{proof}
	
	In particular, Condition~(\ref{eq:PositiveSide}) for $k=0$ states that $\norm{\Delta(y)} \leq c$ for all $y \in \ZZ$.
	
		\begin{no}
		For the rest of this paper, $f$ will denote a $1$\-/quasihomomorphism $\ZZ \to \Sym{n \times n, \ourfield}$ with $f(1)=0$; its delta map $\Delta_f$ will be denoted by $\Delta$. We will denote $\symmspan{\Delta(i)}$ by $L_i$.  Since $\rk{\Delta(i)}\leq 1$, we have that $\dim (\deltaspace i) \leq 1$.
	\end{no}
	
	Note that if $\dim (\sum_{i \in \ZZ}{L_i}) \leq 2$, 
	then by \eqref{eq:findeltapos} and \eqref{eq:findeltaneg} we also have $\rk{f(x)} \leq 2$ for all $x \in \ZZ$, and Theorem \ref{thm:GeneralCase} is true with $A=0$. So from now on we will assume: %

	\begin{assumption} \label{assum:dimSpanGeq3}
		$\dim (\sum_{i \in \ZZ}{L_i}) \geq 3$.
	\end{assumption}
	
	Then we can make the following observation.
	
	\begin{lm} \label{lem:-1plus0}
		If Assumption~\ref{assum:dimSpanGeq3} holds, then $\Delta(0)+\Delta(-1) =0$.
	\end{lm}
	\begin{proof}
		Note that Equation~\eqref{eq:NegativeSide} for $k=1$ tells us that for all $z \in \ZZ$ we have
		\[
		\rk{\Delta(0)+\Delta(-1)-\Delta(z)} \leq 1.
		\]
		By Assumption~\ref{assum:dimSpanGeq3}, we can apply Corollary~\ref{cor:symMatrixRank2} to conclude that
		$\Delta(0)+\Delta(-1) = 0$. 
	\end{proof}
	\begin{observation} \label{obs:symmetryOfAssumption} 
		Still working under Assumption \ref{assum:dimSpanGeq3},
		now Equation~\eqref{eq:NegativeSide} for $k \ge 0$ can be rewritten as 
		\begin{equation} \label{eq:NegativeSideShort}
			\rk{\sum_{i=2}^{k+1} \Delta(-i) - \sum_{i=0}^{k} \Delta(y-i)}  \le 1.
		\end{equation}
		Note the symmetry: if we define $\tilde{\Delta}(x) := \Delta(-1-x)$, then $\Delta$ satisfies the assumptions (\ref{eq:PositiveSide}) and (\ref{eq:NegativeSideShort}) if and only if $\tilde{\Delta}$ does. 
	\end{observation}

		\begin{figure}[H]
			\centering
			\begin{minipage}{.5\textwidth}
				\resizebox{\textwidth}{!}{
					%% Creator: Inkscape 1.2 (dc2aeda, 2022-05-15), www.inkscape.org
%% Accompanies image file 'Zeichnung.pdf' (pdf, eps, ps)
\begingroup%
  \makeatletter%
  \providecommand\color[2][]{%
    \errmessage{(Inkscape) Color is used for the text in Inkscape, but the package 'color.sty' is not loaded}%
    \renewcommand\color[2][]{}%
  }%
  \providecommand\transparent[1]{%
    \errmessage{(Inkscape) Transparency is used (non-zero) for the text in Inkscape, but the package 'transparent.sty' is not loaded}%
    \renewcommand\transparent[1]{}%
  }%
  \providecommand\rotatebox[2]{#2}%
  \newcommand*\fsize{\dimexpr\f@size pt\relax}%
  \newcommand*\lineheight[1]{\fontsize{\fsize}{#1\fsize}\selectfont}%
  \ifx\svgwidth\undefined%
    \setlength{\unitlength}{488.68939257bp}%
    \ifx\svgscale\undefined%
      \relax%
    \else%
      \setlength{\unitlength}{\unitlength * \real{\svgscale}}%
    \fi%
  \else%
    \setlength{\unitlength}{\svgwidth}%
  \fi%
  \global\let\svgwidth\undefined%
  \global\let\svgscale\undefined%
  \makeatother%
  \begin{picture}(1,0.12883781)%
    \lineheight{1}%
    \setlength\tabcolsep{0pt}%
    \put(0,0){\includegraphics[width=\unitlength,page=1]{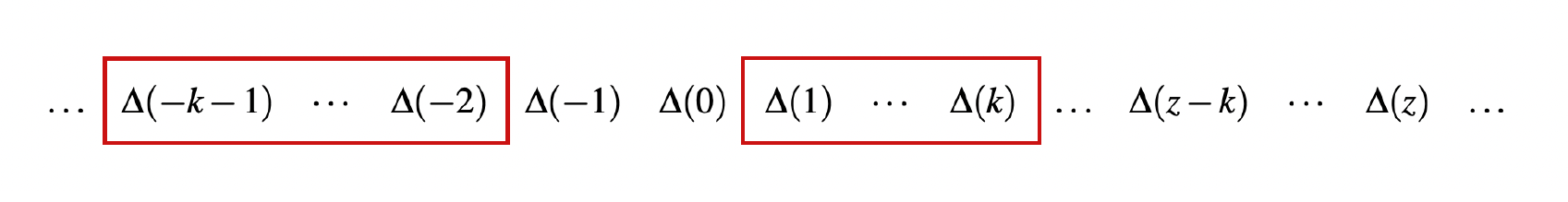}}%
    \put(0.17812373,0.00432838){\makebox(0,0)[lt]{\lineheight{1.25}\smash{\begin{tabular}[t]{l}$k$\end{tabular}}}}%
    \put(0.54501906,0.01028453){\makebox(0,0)[lt]{\lineheight{1.25}\smash{\begin{tabular}[t]{l}$k$\end{tabular}}}}%
    \put(0,0){\includegraphics[width=\unitlength,page=2]{Zeichnung.pdf}}%
    \put(0.80708711,0.00790209){\makebox(0,0)[lt]{\lineheight{1.25}\smash{\begin{tabular}[t]{l}$k+1$\end{tabular}}}}%
  \end{picture}%
\endgroup%

				}
				\caption{\normalfont Equation~\eqref{eq:PositiveSide} says that the sum of the right red block and the sum of the blue block differ by a rank one matrix. Similarily, Equation~\eqref{eq:NegativeSide} says that the sum of the left red block and the sum of the blue block differ by a rank one matrix.}
			\end{minipage}  
		\end{figure}
		
		Next, note that if $\dim (\sum_{i \in \ZZ}{L_i}) \geq 3$ but $\dim (\sum_{i \in \ZZ\setminus\{0,-1\}}{L_i}) \leq 2$, it still holds that $\rk{f(x)} \leq 2$ for all $x \in \ZZ$. So we will replace Assumption~\ref{assum:dimSpanGeq3} with something slightly stronger:
		
		\begin{assumption} \label{assum:dimSpanGeq3bis}
			$\dim (\sum_{i \in \ZZ\setminus\{0,-1\}}{L_i}) \geq 3$.
		\end{assumption}
		
		Under this assumption, we will show that $\Delta$ needs to have a very specific structure.

		\section{Palindromicity}
		\label{sec:Palindromicity}
		Now we show that $\Delta$ satisfies a property reminiscent of palindromes.
		\begin{no}
			For $m \in \NN$, write \begin{align} \label{eq:DefinitionVm} V_m = L_{-m-1} + \cdots +  L_{-2} + L_1 + \ldots +L_m. \end{align}
			Note that $L_{-1}, L_{0}$ are not part of the sum. Assumption~\ref{assum:dimSpanGeq3bis} precisely says that there exists an $m$ with $\dim V_m \geq 3$.
		\end{no}

		\begin{lm} \label{lm:palindromic}
			Let $m$ be such that $V_{m} \supsetneq V_{m-1}$.
			\begin{enumerate}
				\item \label{itemA}  For all $i \in \{1,\ldots,m-1\}$ we have that
				\begin{equation} \label{eq:palindromic}
					\Delta(i)=\Delta(m-i)=\Delta(-i-1)=\Delta(i-m-1).
				\end{equation}	
				\item \label{itemB}  Moreover, if $\dim V_m \geq 3$, it holds that 
				\[\Delta(m+1)=-\Delta(m) \quad \text{ and } \quad \Delta(-m-2)=-\Delta(-m-1).\]
				In particular, $L_m = L_{m+1}$ and $L_{-m-2} = L_{-m-1}$.
			\end{enumerate}  
		\end{lm}
		
		\begin{re}
			\label{re:VisualizationSequence}
			To state Lemma \ref{lm:palindromic} more visually: if $V_{m} \supsetneq V_{m-1}$ and $\dim V_m \geq 3$, then $\Delta$ has the following structure: 
			
			\begin{center}
				\setlength{\tabcolsep}{0.4em}
				\scalebox{0.8}{
					\begin{tabular}{c|ccccccccccccccccccccc}
						$i$&$-m-2$&$-m-1$&&&$\cdots$&&&$-1$&$0$&&&$\cdots$&&&$m$&$m+1$\\
						\hline
						$\Delta(i)$ & \red{$\alpha$} & \red{$-\alpha$} & \blue{$a$} & \blue{$b$} & \blue{$\cdots$} & \blue{$b$} & \blue{$a$} & \red{$\beta$} & \red{$-\beta$} & \blue{$a$} & \blue{$b$} & \blue{$\cdots$} & \blue{$b$} & \blue{$a$} & \red{$\gamma$} & \red{$-\gamma$}
					\end{tabular}.
				}
			\end{center}
		\end{re}
		\begin{proof}
			For Item~(\ref{itemA}) we show $3$ equalities for $i \in \aset{1, \dots, m-1}$:
			\begin{itemize}
				\item  $\Delta(i)=\Delta(m-i)$, which encodes palindromicity of the right blue block; 
				\item  $\Delta(m-i)=\Delta(-i-1)$, which encodes equality of the blocks; 
				\item   $\Delta(-i-1) = \Delta(i-m-1)$, which encodes palindromicity of the left blue block.
			\end{itemize}
			Note that the third equality follows from the first two by substituting $i$ for $m-i$.
			By symmetry (cfr.\ Observation \ref{obs:symmetryOfAssumption}) we may assume that $L_m \not \subset V_{m-1}$. 
			
			We first prove the identity $\Delta(i)=\Delta(m-i)$ by induction on $i$. For the base case $i=1$, observe that setting $k = 1$ and  $z = m$ 
			in Equation~\eqref{eq:PositiveSide} gives 
			\begin{align*}
				&	\rk{\Delta(1) - \Delta(m-1) - \Delta(m)} \leq 1.
			\end{align*}
			By Corollary~\ref{cor:ABmatrices}
			we get that $\Delta(1)=\Delta(m-1)$. 
			For the case $i=2$, we put $k = 2$ and  $z = m$ in Equation~\eqref{eq:PositiveSide}:
			\begin{align*}
				&	\rk{\Delta(1) + \Delta(2) - \Delta(m-2) - \Delta(m-1) - \Delta(m)} \leq 1.
			\end{align*}
			Using $\Delta(1)=\Delta(m-1)$ and Corollary~\ref{cor:ABmatrices} we find $\Delta(2)=\Delta(m-2)$.
			
			One proceeds in a similar fashion for higher $i$. Namely, if the equality is true for $i$, one gets the equality for $i+1$ from Equation~\eqref{eq:PositiveSide} 
			by setting $k= i + 1$ and $z = m$.
			The equality $\Delta(m-i)=\Delta(-i-1)$ is proven analogously, using Equation~\eqref{eq:NegativeSideShort}. 
			
			For Item~(\ref{itemB}), we want to show that $\Delta(m+1)+\Delta(m)=0$. 
			If $i$ is in $\aset{1,\ldots m-1,m}$, Equation~\eqref{eq:PositiveSide} for $z=m+1$ and  $k=i$, 
			combined with (\ref{eq:palindromic}), imply that 
			\[
			\rk{\Delta(m+1)+\Delta(m)-\Delta(i)} \leq 1.
			\] 
			When $i$ is in $\aset{-m-1,-m,\ldots,-2}$  the same equation can be derived from Equation~\eqref{eq:NegativeSideShort} for $z=m+1$ and $k=-i-1$.
			Since $\dim V_m \geq 3$, 
			by Corollary~\ref{cor:symMatrixRank2}
			this implies that $\Delta(m+1)+\Delta(m) = 0$. 
			The proof that $\Delta(-m-2)+\Delta(-m-1)=0$ is analogous.
			Finally we have that $L_m = \symmspan {\Delta(m)} = \symmspan{\Delta(m+1)} =  L_{m+1}$, and analogously for the other one.
		\end{proof}

		\section{\DQ{} sequences}
		Now, our aim is to show that the finite pattern observed in Section~\ref{sec:Palindromicity}
		can be extended to infinity. We call a sequence satisfying this pattern \DQ{}, meaning \emph{almost periodic almost palindromic}. In this section, we define \DQ{} sequences and prove some general lemmas; in the next section we will show that our delta sequence is \DQ{}. For the purposes of this section, $H$ can be any abelian group.
		
		\begin{de} 
			\label{de:AlmostPeriodicAlmostPalindromic}
			A sequence $\Seq{\Delta(i)}_{i=-N}^{N-1}$, with $\Delta(i) \in H$ 
			is \DQ{} with period $p\in[2,N]$,
			if
			\begin{align}
				\Delta({i+p}) &= \Delta(i) &&\text{ if } i \not \equiv -1 \text{ or } 0 \mod p, \label{eq:APAPperiod} \\
				\Delta({j-1}) + \Delta({j}) &= 0 && \forall j \in \{-N+1, \ldots, N-1\} \text{ with } p|j, \label{eq:APAPcancel}\\
				\Delta({p-1-i}) &= \Delta(i) && \forall i = 1, \dots, p-2. \label{eq:APAPpalin}
			\end{align}
			
			From now on we will refer to the respective Conditions~(\ref{eq:APAPperiod}), (\ref{eq:APAPcancel}), (\ref{eq:APAPpalin}).
			We call $\Delta(1), \dots, \Delta({p-3}), \Delta({p-2})$ the \emph{palindromic block}, and will  write $\palinsum{\Delta}$ for the ``block sum" $\Delta(1) + \cdots + \Delta(p-2)$.

		\end{de}
		
		\begin{re} 
			\label{re:IllustrationAPAPdefinition}
			The next two pictures illustrate how an \DQ{} sequence looks like. First we see a global picture:

			\begin{figure}[H]
				\centering
				\begin{minipage}{0.5\textwidth}
					\resizebox{\textwidth}{!}{
						%% Creator: Inkscape 1.2 (dc2aeda, 2022-05-15), www.inkscape.org
%% Accompanies image file 'globalpicture.pdf' (pdf, eps, ps)
\begingroup%
  \makeatletter%
  \providecommand\color[2][]{%
    \errmessage{(Inkscape) Color is used for the text in Inkscape, but the package 'color.sty' is not loaded}%
    \renewcommand\color[2][]{}%
  }%
  \providecommand\transparent[1]{%
    \errmessage{(Inkscape) Transparency is used (non-zero) for the text in Inkscape, but the package 'transparent.sty' is not loaded}%
    \renewcommand\transparent[1]{}%
  }%
  \providecommand\rotatebox[2]{#2}%
  \newcommand*\fsize{\dimexpr\f@size pt\relax}%
  \newcommand*\lineheight[1]{\fontsize{\fsize}{#1\fsize}\selectfont}%
  \ifx\svgwidth\undefined%
    \setlength{\unitlength}{505.3169614bp}%
    \ifx\svgscale\undefined%
      \relax%
    \else%
      \setlength{\unitlength}{\unitlength * \real{\svgscale}}%
    \fi%
  \else%
    \setlength{\unitlength}{\svgwidth}%
  \fi%
  \global\let\svgwidth\undefined%
  \global\let\svgscale\undefined%
  \makeatother%
  \begin{picture}(1,0.02624107)%
    \lineheight{1}%
    \setlength\tabcolsep{0pt}%
    \put(0,0){\includegraphics[width=\unitlength,page=1]{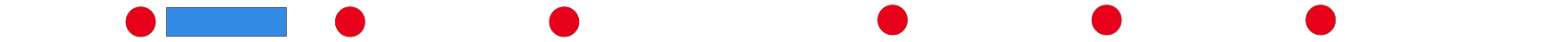}}%
    \put(0.88880382,0.00644383){\makebox(0,0)[lt]{\lineheight{1.25}\smash{\begin{tabular}[t]{l}$\cdots$\end{tabular}}}}%
    \put(0.42369815,0.00453477){\makebox(0,0)[lt]{\lineheight{1.25}\smash{\begin{tabular}[t]{l}$\cdots$\end{tabular}}}}%
    \put(0,0){\includegraphics[width=\unitlength,page=2]{globalpicture.pdf}}%
    \put(-0.00194706,0.00538102){\makebox(0,0)[lt]{\lineheight{1.25}\smash{\begin{tabular}[t]{l}$\cdots$\end{tabular}}}}%
    \put(0,0){\includegraphics[width=\unitlength,page=3]{globalpicture.pdf}}%
  \end{picture}%
\endgroup%

					}
				\end{minipage}  
			\end{figure}
			
			The blue box represents the palindromic block, whereas the red circles represent the $p$-cancellation. Eack box has length $p-2$.
			Note that while the blue box is always meant to be the same, the red circles are not. %
			
			Next, we see the same picture but now zoomed in: \vspace{0.2cm} %
			\begin{center}
				\setlength{\tabcolsep}{0.4em}
				\scalebox{0.8}{
					\begin{tabular}{c|ccccccccccccccccccccc}
						$i$&$(a-1)p-1$&$(a-1)p$&&&$\cdots$&&&$ap-1$\\
						\hline
						$\Delta(i)$ & \red{$\alpha$} & \red{$-\alpha$} & \blue{$a$} & \blue{$b$} & \blue{$\cdots$} & \blue{$b$} & \blue{$a$} & \red{$\beta$}
				\end{tabular}}
			\end{center}
			
			\vspace{0.5cm}
			
			\begin{center}
				\setlength{\tabcolsep}{0.4em}
				\scalebox{0.8}{
					\begin{tabular}{c|ccccccccccccccccccccc}
						$i$&$ap$&&&$\cdots$&&&$(a+1)p-1$&$(a+1)p$\\
						\hline
						$\Delta(i)$ & 
						\red{$-\beta$} & \blue{$a$} & \blue{$b$} & \blue{$\cdots$} & \blue{$b$} & \blue{$a$} & \red{$\gamma$} & \red{$-\gamma$}
					\end{tabular}.}
			\end{center}
			\vspace{0.2cm}
			In this picture we see the cancellation in red and the palindromic block in blue.
		\end{re}

		The following result is a quick calculation that uses the three properties of being an \DQ{} sequence.
		
		\begin{lm} 
			\label{lm:AddingTermsAPAPseq}
			Let $\Seq{\Delta(i)}_{i=-N}^{N-1}$ be an \DQ{} sequence with period $p$.
			For any $k \in \ZZ_{\ge 0}$, the sum of any $kp$ consecutive elements in $\Seq{\Delta(i)}_{i=-N}^{N-1}$, where the index of the first element is not a multiple of $p$ %
			is constant.
			Moreover, this constant equals $k \cdot B_{\Delta}$. \qed %
		\end{lm}

		Our first source of \DQ{} sequences is Lemma~\ref{lm:palindromic}:
		
		\begin{lm} 
			\label{lm:APAPdivides}
			Let $m$ be such that $V_{m} \supsetneq V_{m-1}$ and $\dim V_m \ge 3$.
			The sequence $\Seq{\Delta(i)}_{-m-2}^{m+1}$ is \DQ{} with period $m + 1$.
			Moreover, for any other period $p$ that makes this sequence \DQ{}
			we have that~$p | m + 1$.
		\end{lm}
		
		\begin{proof}
			Since $\dim V_m \ge 3$, the sequence $\Seq{\Delta(i)}_{-m-2}^{m+1}$ is \DQ{} with period $m+1$
			by the two items of Lemma~\ref{lm:palindromic}.
			Now suppose that $\Seq{\Delta(i)}_{-m-2}^{m+1}$ is \DQ{} with period $p$.
			Since $V_{m} \supsetneq V_{m-1}$ at least one of $L_m \not \subset V_{m-1}$ or $L_{-m-1} \not \subset V_{m-1}$ is true. 
			By the symmetry from Observation~\ref{obs:symmetryOfAssumption} we assume the former.

			Suppose that $p$ does not divide $m+1$, so $m \not \equiv -1 \mod p$.
			If additionally we have that $m \not \equiv 0 \mod p$, 
			then $\Delta(m) = \Delta(j)$ with $j$ the residue of $m$ divided by~$p$.
			Since  $j < p < m+1$, we get  $\Delta(m) = \Delta(j)$ is in $V_{m-1}$, a contradiction.
			To finish, assume that $m \equiv 0 \mod p$, so then 
			$\Delta(m-1) + \Delta(m) = 0$, which again implies that $\Delta(m)$ is in $V_{m-1}$, a contradiction.
		\end{proof}

		Next, we use the last claim from Lemma~\ref{lm:APAPdivides} to study how two distinct \DQ{} structures on the same sequence interact.
		We apply this result in Claim~\ref{claim:SmallerPeriod}.

		\begin{lm} \label{lemma:DQgcd}
			Let $(\Delta(i))_{i=-N}^{N-1}$ 
			be \DQ{} with period $p$. Suppose there is a $q \in [2,p-1]$ such that
			\begin{enumerate}
				\item\label{it:qperiod} $\Delta(i) = \Delta(i+q) \text{ for } i=1,\ldots, p-q-2$ ($q$-periodicity), 
				\item\label{it:qpalin} $\Delta(i) = \Delta(q-1-i) \text{ for } i=1,\ldots, q-2$ (palindromicity of the first $q-2$ elements),
				\item\label{it:qcancel} $\Delta(q-1) + \Delta(q) = 0$,
			\end{enumerate}
			Write $g=\gcd(p,q)$. If $g>1$ then $\Delta$ is \DQ{} with period $g$. If $g=1$ then all $\Delta(i)$ are the same up to a sign. %
		\end{lm}
		
		We will deduce this using the following easy number-theoretic lemma: 
		\begin{lm} \label{lemma:equivalenceRelationModulo}
			Let $q<p$ be integers and write $g=\gcd(p,q)$. %
			Consider the equivalence relation $\sim$ on $\ZZ$ generated by:
			\begin{itemize}
				\item $x \sim y$ if $x \equiv y \mod q$ ($q$-periodic),
				\item $x \sim q-1-x$ for $x$ in $\{ 0,\ldots, q-1 \}$ ($q$-palindromic),
				\item $x \sim p-1-x$ for $x$ in $\{0,\ldots, p-1 \}$ ($p$-palindromic).
			\end{itemize}
			Then we have that $x \sim y$ if and only if $x \equiv y \mod g$ or $x+y \equiv -1 \mod g$.
		\end{lm}
		
		\begin{proof}[Proof of Lemma \ref{lemma:equivalenceRelationModulo}]
			We first show $\sim$ is also $p$-periodic. 
			For this, take any $x \in \ZZ$, and let $m \in \ZZ$ be the unique integer for which $p-q \leq x-mq \leq p-1$. Indeed,
			\begin{align*}
				x &\sim x-mq \sim p-1-(x-mq) \sim q-1 - (p-1-x+mq)\\
				&= x - p - (m-1)q \sim x-p. 
			\end{align*} 
			In the previous calculation, $x-mq$ is contained in $\{0,\ldots, p-1 \}$ and $ p-1-(x-mq)$ is in $\{ 0,\ldots, q-1 \}$, so the operations are valid.
			The combination of $q$-periodicity and $p$-periodicity is equivalent to $g$-periodicity, 
			namely $x \sim y$ when $x \equiv y \mod g$.
			Additionally, palindromicity gives $x \sim y$ when $x + y \equiv -1 \mod g$.
			Indeed, by $g$-periodicity we may assume that $x$ is in $\{1, \dots, g-1 \}$, then
			by $q$-palindromicity and periodicity we have that $x \sim q - 1 -x \sim -1 -x \sim y$.
		\end{proof}
		
		\begin{proof}[Proof of Lemma \ref{lemma:DQgcd}]
			We first consider the case $g>1$. 
			Let us write $\class{a}{q}$
			for the unique integer in $\{1,\ldots,q\}$ that is congruent to $a$ modulo $q$.  
			\begin{cl} \label{claim:FirstQ}
				It suffices to check the APAP property on the interval $[1,q]$. In other words: if we verify the identities
				\begin{enumerate}[label=(\alph*)]
					\item\label{it:gperiod} $\Delta(i)=\Delta(i+g)$ for $i\in \{1,\ldots,q-g-2\}$ with $i \not \equiv -1 \text{ or } 0 \mod g$,
					\item\label{it:gcancel} $\Delta(kg-1)+\Delta(kg) = 0$ for $k=1,\ldots,q/g$,
					\item\label{it:gpalin} $\Delta(i)=\Delta(g-1-i)$ for $i=1,\ldots,g-2$,
				\end{enumerate}
				then $\Delta$ is \DQ{} with period $g$.
			\end{cl}
			\begin{proof}
				For palindromicity there is nothing to prove. For periodicity: given any $i \neq -1 \text{ or } 0 \mod g$, we have 
				\[
				\Delta(i)=\Delta(\class{i}{p})=\Delta(\class{\class{i}{p}}{q})=\Delta(\class{i}{g}),
				\]
				where we used $p$-periodicity, $q$-periodicity, and \ref{it:gperiod}.
				
				Cancellation is similar: if $p | j$ then $\Delta(j)=-\Delta(j-1)$ by $p$\-/cancellation; if $g \mid j$ but $p\nmid j$ then we can use $p$-periodicity, $q$-periodicity, and \ref{it:gcancel} to find
				\begin{align*}
					\Delta(j-1)+\Delta(j)&=\Delta(\class{j-1}{p})+\Delta(\class{j}{p})\\
					&=\Delta(\class{\class{j-1}{p}}{q})+\Delta(\class{\class{j}{p}}{q})=0. \qedhere
				\end{align*}
			\end{proof}
			We now verify the conditions \ref{it:gperiod}, \ref{it:gcancel}, \ref{it:gpalin} above. 
			For this, we formally define the $q$-periodic map $\altDelta : \ZZ \to H$ by $\altDelta(i) = \altDelta(\class{i}{q})$.  
			Since $\Delta$ and  $\altDelta$ agree on the interval $[1,q]$, by Claim~\ref{claim:FirstQ} we may work now with $\altDelta$ instead.
			We consider the equivalence relation $\sim$ from the previous lemma. 
			Then showing \ref{it:gperiod} and \ref{it:gpalin} amounts to showing that $\altDelta$ is constant on every equivalence class except for the one generated by $0$. 
			Indeed, two numbers $x$ and $y$ in the same equivalence class can be connected by a chain as in Lemma \ref{lemma:equivalenceRelationModulo}, and the only case this doesn't imply an equality of $\altDelta$ is when $x=0$, $q-1$, or $p-1$, but then we are in the bad equivalence class. 
			
			We are left with showing \ref{it:gcancel}. For this, we in fact will prove the stronger claim that $\altDelta(kg-1)=\altDelta(kg) = 0$ for $k=1,\ldots,q/g$. Viewing $\altDelta$ as a map $\ZZ / {q\ZZ} \to H$, we claim that 
			\begin{align} \label{eq:TildeDeltaPalindrome}
				\altDelta(i)=\altDelta(p-1-i) 
			\end{align}
			for every $i \in \ZZ / {q\ZZ}$. The only nontrivial case is $i=0$: 
			if $q=p-1$ then $\altDelta(0)= \altDelta(q) = \altDelta(p-1)$, and if $q<p-1$ then by $q$-periodicity and $p$-palindromicity we get $\altDelta(0)=\altDelta(q)=\altDelta(p-1-q)=\altDelta(p-1)$. 
			
			This naturally leads us to the sequence
			\[ 
			\altDelta(0), \altDelta(p-1), \altDelta(-p), \ldots, \altDelta(-(k-1)p), \altDelta(kp-1), \altDelta(-kp), \ldots
			\]
			Besides having $\altDelta(-(k-1)p)= \altDelta(kp-1)$ by Equation~\eqref{eq:TildeDeltaPalindrome}, 
			we also have $\altDelta(kp-1) = \altDelta(-kp)$ by $\Delta(i)$ being \DQ{}  with period $p$, except when $[kp]_q-1=0$ or $[kp]_q=0$. 
			Since $g > 1$, we never have $[kp]_q=1$. 
			Thus, we let $b=\frac{q}{g}$ be the smallest natural number such that $[bp]_q=0$, so we get
			\[
			\altDelta(0) = \altDelta(p-1) = \altDelta(-p) = \cdots = \altDelta(-(b-1)p) = \altDelta(bp-1). %
			\]
			Note that the set of arguments in the above chain of equalities contains every $x \in \ZZ / {q\ZZ}$ that is congruent to $0$ or $-1$ modulo $g$.
			Moreover, since $\Delta(q-1) + \Delta(q) = 0$, we get $\altDelta(bp-1) = -\altDelta(0)$, rendering the whole sequence equal to 0, as desired.

			Now assume that $g$ equals 1. By $q$-periodicity it suffices to show that $\Delta(1), \ldots, \Delta(q)$ are equal up to a sign. 
			We define $\altDelta$ as above, and let $a$ be the smallest natural number such that $[ap]_q=1$. 
			Then, by similar arguments, we find that
			\begin{align*}
				\altDelta(0) &= \altDelta(p-1) = \altDelta(-p) = \cdots = \altDelta(-(a-1)p) = \\ 
				&= \altDelta(ap-1) = - \altDelta(ap) = \cdots \\
				&= -\altDelta(qp-1) = \altDelta(0).
			\end{align*} 
			Note that above, since $g=1$ we have that $[ap]_q = 0$ for the first time when $a=q$.
			So we find that all the $\altDelta(x)$, for $x \in \ZZ / {q\ZZ}$, are equal up to a sign, as desired. 
		\end{proof}
		
		\section{The Delta sequence is \DQ{}}
		
		In the following theorem we use the same notation as before, i.e. given a 1\-/quasihomomorphism $f$ we denote by $L_i$ the space $\symmspan{ \Delta_f(i)}$.
		
		\begin{thm} \label{thm:APAP}
			Let $f: \ZZ \to \Sym{n\times n,\ourfield}$ be a 1\-/quasihomomorphism. 
			Assume that $\dim (\sum_{i \in \ZZ\setminus\{0,-1\}}{L_i}) \geq 3$.	
			We can find a natural number $p$ such that $\Delta_f$ is \DQ{} with period $p$. Moreover, $p$ can be chosen such that $\dim(L_1 + \cdots + L_{p-2}) \leq 2$; hence in particular $\rk{\palinsum{\Delta}} \leq 2$. %
		\end{thm}
		
		\begin{proof}[Proof of Theorem \ref{thm:APAP}]
			Let $m$ be minimal such that $\dim V_m >2$. 
			By Lemma~\ref{lm:palindromic} we have that the sequence $\Seq{\Delta(i)}_{i=-m-2}^{m+1}$ %
			is~\DQ{} with period~$m+1$. Let $p$ be the minimal positive integer such that $\Seq{\Delta(i)}_{i=-m-2}^{m+1}$ is~\DQ{} with period~$p$. By Lemma \ref{lm:APAPdivides} we have that $p$ is a divisor of $m+1$. 
			We will show that the entire sequence $\Seq{\Delta(i)}_{i=-\infty}^{\infty}$ is \DQ{} with period~$p$. 
			Then, using minimality of $m$, we get that %
			\[\dim(L_1 + \cdots + L_{p-2}) = \dim V_{p-2} \le 2,\]
			which implies that $\rk{\palinsum{\Delta}} \leq 2$.

			Now, assume that for some $N$ the sequence $\Seq{\Delta(i)}_{i=-N}^{N-1}$ is \DQ{} with period~$p$. 
			We will simultaneously extend the sequence by one on both sides, and show $\Seq{\Delta(i)}_{i=-N-1}^{N}$ is still \DQ{} with period $p$; thus proving the theorem by induction. 
			
			We have three cases.
			If $N \equiv -1 \mod p$ there is nothing to prove, as illustrated in the following picture.
			
			\begin{figure}[H]
				\centering
				\begin{minipage}{0.5\textwidth}
					\resizebox{\textwidth}{!}{
						%% Creator: Inkscape 1.2 (dc2aeda, 2022-05-15), www.inkscape.org
%% Accompanies image file 'Case1.pdf' (pdf, eps, ps)
\begingroup%
  \makeatletter%
  \providecommand\color[2][]{%
    \errmessage{(Inkscape) Color is used for the text in Inkscape, but the package 'color.sty' is not loaded}%
    \renewcommand\color[2][]{}%
  }%
  \providecommand\transparent[1]{%
    \errmessage{(Inkscape) Transparency is used (non-zero) for the text in Inkscape, but the package 'transparent.sty' is not loaded}%
    \renewcommand\transparent[1]{}%
  }%
  \providecommand\rotatebox[2]{#2}%
  \newcommand*\fsize{\dimexpr\f@size pt\relax}%
  \newcommand*\lineheight[1]{\fontsize{\fsize}{#1\fsize}\selectfont}%
  \ifx\svgwidth\undefined%
    \setlength{\unitlength}{501.29732238bp}%
    \ifx\svgscale\undefined%
      \relax%
    \else%
      \setlength{\unitlength}{\unitlength * \real{\svgscale}}%
    \fi%
  \else%
    \setlength{\unitlength}{\svgwidth}%
  \fi%
  \global\let\svgwidth\undefined%
  \global\let\svgscale\undefined%
  \makeatother%
  \begin{picture}(1,0.0952055)%
    \lineheight{1}%
    \setlength\tabcolsep{0pt}%
    \put(0,0){\includegraphics[width=\unitlength,page=1]{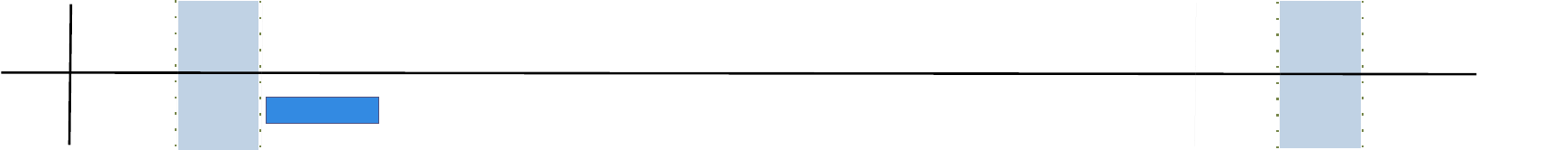}}%
    \put(0.12078428,0.06572535){\makebox(0,0)[lt]{\lineheight{1.25}\smash{\begin{tabular}[t]{l}-N-1\end{tabular}}}}%
    \put(0.83623286,0.06427417){\makebox(0,0)[lt]{\lineheight{1.25}\smash{\begin{tabular}[t]{l}N\end{tabular}}}}%
    \put(0.46966538,0.06401271){\makebox(0,0)[lt]{\lineheight{1.25}\smash{\begin{tabular}[t]{l}$\cdots$\end{tabular}}}}%
    \put(0.46922655,0.01687648){\makebox(0,0)[lt]{\lineheight{1.25}\smash{\begin{tabular}[t]{l}$\cdots$\end{tabular}}}}%
    \put(0.88791219,0.06244877){\makebox(0,0)[lt]{\lineheight{1.25}\smash{\begin{tabular}[t]{l}$\cdots$\end{tabular}}}}%
    \put(0.06780621,0.06386793){\makebox(0,0)[lt]{\lineheight{1.25}\smash{\begin{tabular}[t]{l}$\cdots$\end{tabular}}}}%
    \put(0.01515825,0.06346999){\makebox(0,0)[lt]{\lineheight{1.25}\smash{\begin{tabular}[t]{l}$i$\end{tabular}}}}%
    \put(-0.00196268,0.01553092){\makebox(0,0)[lt]{\lineheight{1.25}\smash{\begin{tabular}[t]{l}$\Delta(i)$\end{tabular}}}}%
    \put(0,0){\includegraphics[width=\unitlength,page=2]{Case1.pdf}}%
  \end{picture}%
\endgroup%

					}
					\caption{\normalfont Here we see that for $i=-N$ and $i=N-1$ we start and end with the palindromic block. Our Definition of APAP is not dependent of what entry we put next for $i=N$ and $i=-N-1$.} %
			\end{minipage}  
		\end{figure}
		
		Next, assume that $N \equiv 0 \mod p$. This case is illustrated as follows: 
		\begin{figure}[H]
			\centering
			\begin{minipage}{0.5\textwidth}
				\resizebox{\textwidth}{!}{
					%% Creator: Inkscape 1.2 (dc2aeda, 2022-05-15), www.inkscape.org
%% Accompanies image file 'Case2.pdf' (pdf, eps, ps)
\begingroup%
  \makeatletter%
  \providecommand\color[2][]{%
    \errmessage{(Inkscape) Color is used for the text in Inkscape, but the package 'color.sty' is not loaded}%
    \renewcommand\color[2][]{}%
  }%
  \providecommand\transparent[1]{%
    \errmessage{(Inkscape) Transparency is used (non-zero) for the text in Inkscape, but the package 'transparent.sty' is not loaded}%
    \renewcommand\transparent[1]{}%
  }%
  \providecommand\rotatebox[2]{#2}%
  \newcommand*\fsize{\dimexpr\f@size pt\relax}%
  \newcommand*\lineheight[1]{\fontsize{\fsize}{#1\fsize}\selectfont}%
  \ifx\svgwidth\undefined%
    \setlength{\unitlength}{501.29732238bp}%
    \ifx\svgscale\undefined%
      \relax%
    \else%
      \setlength{\unitlength}{\unitlength * \real{\svgscale}}%
    \fi%
  \else%
    \setlength{\unitlength}{\svgwidth}%
  \fi%
  \global\let\svgwidth\undefined%
  \global\let\svgscale\undefined%
  \makeatother%
  \begin{picture}(1,0.0952055)%
    \lineheight{1}%
    \setlength\tabcolsep{0pt}%
    \put(0,0){\includegraphics[width=\unitlength,page=1]{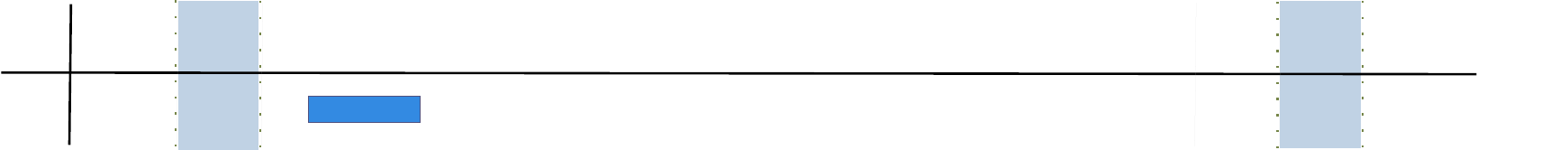}}%
    \put(0.12078428,0.06572535){\makebox(0,0)[lt]{\lineheight{1.25}\smash{\begin{tabular}[t]{l}-N-1\end{tabular}}}}%
    \put(0.83623286,0.06427417){\makebox(0,0)[lt]{\lineheight{1.25}\smash{\begin{tabular}[t]{l}N\end{tabular}}}}%
    \put(0.46966538,0.06401271){\makebox(0,0)[lt]{\lineheight{1.25}\smash{\begin{tabular}[t]{l}$\cdots$\end{tabular}}}}%
    \put(0.46922655,0.01687648){\makebox(0,0)[lt]{\lineheight{1.25}\smash{\begin{tabular}[t]{l}$\cdots$\end{tabular}}}}%
    \put(0.88791219,0.06244877){\makebox(0,0)[lt]{\lineheight{1.25}\smash{\begin{tabular}[t]{l}$\cdots$\end{tabular}}}}%
    \put(0.06780621,0.06386793){\makebox(0,0)[lt]{\lineheight{1.25}\smash{\begin{tabular}[t]{l}$\cdots$\end{tabular}}}}%
    \put(0.01515825,0.06346999){\makebox(0,0)[lt]{\lineheight{1.25}\smash{\begin{tabular}[t]{l}$i$\end{tabular}}}}%
    \put(-0.00196268,0.01553092){\makebox(0,0)[lt]{\lineheight{1.25}\smash{\begin{tabular}[t]{l}$\Delta(i)$\end{tabular}}}}%
    \put(0,0){\includegraphics[width=\unitlength,page=2]{Case2.pdf}}%
  \end{picture}%
\endgroup%

				}
				\caption{\normalfont Here our $\Delta$ starts for $i=N$ with an "end-cancellation", e.g. $-\alpha$ and it ends with another "start-cancellation", e.g. $\gamma$. Keeping our global picture in mind, we can see that the sum of the entries $\Delta(-N-1)$ and $\Delta(-N)$, resp. $\Delta(N-1)$ and $\Delta(N)$, should be zero.} 	
			\end{minipage}  
			\label{picture:case2}
		\end{figure} 
		
		So we need to show $\Delta(N-1) + \Delta(N) = 0$ and $\Delta(-N-1) + \Delta(-N) = 0$. We reason analogously to the proof of Item~(\ref{itemB}) from Lemma~\ref{lm:palindromic}.
		Equations~\eqref{eq:PositiveSide} and~\eqref{eq:NegativeSideShort} yield 
		\[    \rk{\Delta(i) - \Delta(N-1) - \Delta(N)} \le 1 \]
		for $i$ in $\aset{-N, \dots , N-1} \setminus \aset{-1,0}$.
		Since $\dim V_m >2$, there are three indices $i$ with linearly independent $L_i$, 
		so by Corollary~\ref{cor:symMatrixRank2} we get that $\Delta(N-1) + \Delta(N) = 0$.
		The other equality follows analogously. %
		
		For the last case, assume that $N \not \equiv -1, 0 \mod p$. %

		Let $i$ be the residue of $N$ when dividing by $p$. Now $\Delta(N)$ and $\Delta(-N-1)$ are both in a palindromic block, and we want to show that $\Delta(N) = \Delta(i) = \Delta(-N-1)$. We only prove the first equality, the second one being analogous.%
		
		We will prove that $\Delta(N)=\Delta(i)$ by contradiction in two steps:
		\begin{enumerate}
			\item Suppose that $\Delta(N) \neq \Delta(i)$, then $L_N \not \subset V_2$.
			\item $L_N \not \subset V_2$ leads to a contradiction with minimality of $p$.
		\end{enumerate}
		
		\begin{cl}
			Suppose that $\Delta(N) \neq \Delta(i)$, then $L_N \not \subset V_2$.
		\end{cl}
		\begin{proof} 
			Apply Equation~\eqref{eq:PositiveSide} with $k=m$ and $z=N$ to get:
			\begin{align*}
				\rk{\sum_{j=1}^{m} \Delta(j) - \sum_{j=0}^{m} \Delta(N-j)}  \le 1.
			\end{align*} 
			Rewrite the sum inside the previous expression as
			\begin{align*}
				\left(-\Delta(m+1) + \sum_{j=1}^{m+1} \Delta(j) \right) 
				- \left(\Delta(N) - \Delta(N-m-1) + \sum_{j=1}^{m+1} \Delta(N-j) \right).
			\end{align*}
			Note that by induction hypothesis
			both $\sum_{j=1}^{m+1} \Delta(j)$ and $\sum_{j=1}^{m+1} \Delta(N-j)$ 
			are sums of $m+1$ consecutive elements in an \DQ{} sequence, 
			and recall that $m+1$ is a multiple of $p$,
			so Lemma~\ref{lm:AddingTermsAPAPseq} implies that both sums cancel each other.
			Since $N-m-1 \equiv i \mod p $, we have
			\[    \rk{\Delta(i) + \Delta(m) - \Delta(N) } \le 1. \]
			Note that $L_m \not \subset V_2$ but $L_i \subseteq V_2$.
			So if also $L_N \subseteq V_2$ this would imply $L_m \not \subset \symmspan{\Delta(i)-\Delta(N)}$, but then Corollary \ref{cor:ABmatrices} yields 
			\[ \rk{\Delta(i) + \Delta(m) - \Delta(N) } = \rk{\Delta(i) - \Delta(N)} +1 \geq 2,\] which is a contradiction.
		\end{proof}
		\begin{cl} \label{claim:SmallerPeriod}
			If $L_N \not \subset V_2$, we get a contradiction with the minimality of~$p$.
		\end{cl}
		\begin{proof}
			We write $q=i+1$, where $i$ is still the residue of $N$ modulo $p$.
			We will apply Lemma \ref{lemma:DQgcd} to show that $(\Delta(i))_{i=-N}^{N-1}$ is \DQ{} with period equal to $\gcd(p,q)$. For this, we need to verify the three conditions. 
			
			Write $N=ap+q-1$. We apply Equation~\eqref{eq:PositiveSide} with $k=1$ and $z=N$: 
			\begin{align}
				\rk{\Delta(1) - (\Delta(N) + \Delta(N-1))} \leq 1.
			\end{align}
			Since our sequence is \DQ{} with period $p$, we find that \\ 
			$\Delta(N-1) = \Delta(ap+q-2) = \Delta(q-2)$, hence
			\begin{align}
				\rk{\Delta(1) - (\Delta(N) + \Delta(q-2))} \leq 1.
			\end{align}
			Since $L_N \not \subset V_2$ but $L_1, L_{q-2} \subset V_2$, we can apply Corollary~\ref{cor:ABmatrices} to $A=\Delta(1)-\Delta(q-2)$ and $B=\Delta(N)$ to conclude that $\Delta(1)-\Delta(q-2)=0$.
			Repeating the argument for $k=2,\ldots,q-2$, we find that
			$$\Delta(k) = \Delta(q-1-k) \text{ for } k=1,\ldots, q-2,$$  showing Condition~\eqref{it:qpalin} of Lemma~\ref{lemma:DQgcd}.
			
			For $k=q-1$, we find 
			$$
			\rk{\Delta(q-1) - (\Delta(N) + \Delta(ap))} \leq 1,
			$$
			but now $L_{ap}$ need not be in $V_2$ and we don't get any new information. However, for $k=q$, we get 
			$$
			\rk{\Delta(q-1) + \Delta(q) - (\Delta(N) + \Delta(ap) + \Delta(ap-1))} \leq 1.
			$$
			Now we know that $\Delta(ap) + \Delta(ap-1) = 0$ and conclude that 
			$$
			\Delta(q-1) + \Delta(q) = 0,
			$$ 
			which shows Condition~(\ref{it:qcancel}) of Lemma~\ref{lemma:DQgcd}.
			
			Now we continue with $k=q+1$:
			$$
			\rk{\Delta(q+1) - (\Delta(N) + \Delta(ap-2))} \leq 1.
			$$
			But we know that $\Delta(ap-2) = \Delta(p-2) = \Delta(1)$, and hence we conclude
			$$
			\Delta(q+1)=\Delta(1).
			$$
			We can continue this up to $k=p-2$, and find that
			\begin{align}
				\Delta(k) = \Delta(k-q) \text{ for } k=q+1,\ldots, p-2,
			\end{align}
			which is Condition~(\ref{it:qperiod}) of Lemma~\ref{lemma:DQgcd}. 
			We have verified all conditions, hence it holds that $(\Delta(i))_{i=-N}^{N-1}$ is \DQ{} with period $g := \gcd(p,q)$. 
			Hence, the shorter sequence $(\Delta(i))_{i=-m-2}^{m+1}$ is \DQ{} with period $g$ strictly less than $p$; contradicting our choice of $p$. 
		\end{proof}

		This finishes our induction, and thus the proof.
	\end{proof}

	\section{Proof of the main result}
	
	Putting everything together we get:
	
	\begin{proof}[Proof of Theorem~\ref{thm:GeneralCase}]
		By Theorem~\ref{thm:APAP} we find that $\Delta$ is \DQ{} with period~$p$. 
		Define 
		\[A = \dfrac{\palinsum{\Delta}}{p}=\dfrac{\Delta(1) + \dots + \Delta(p-2)}{p} = \dfrac{f(p-1)}{p}. \]
		We will show that Equation~\eqref{eq:mainTheorem} holds with this $A$. We restrict to the case $x \geq 1$; the other case being analogous.
		Write $x = ap + r$ with $1 \le r \le p$.
		If $x \geq 1$,%
		we have that  $f(x) = \Delta(1) + \dots + \Delta(x-1)$.
		Applying Lemma~\ref{lm:AddingTermsAPAPseq} we get that:
		\begin{align} 
			\label{eq:SubstitutingSymmetry}
			\nonumber     f(x) = &\Delta(1) + \dots + \Delta(x-1)  \\
			\nonumber         = &aB_{\Delta} + \sum_{j=1}^{r-1} \Delta(ap + j)  \\
			= &apA + \sum_{j=1}^{r-1} \Delta(ap + j). 
		\end{align}
		
		We have two cases. 
		First, we assume that $r = p$. 
		Equation~\eqref{eq:SubstitutingSymmetry} becomes
		\begin{align*} 
			f(x) &= apA + \sum_{j=1}^{p-1} \Delta(ap + j)  \\
			&= apA + \sum_{j=1}^{p-2} \Delta(ap + j) + \Delta(x-1)\\
			&= apA + pA + \Delta(x-1) = xA + \Delta(x-1).
		\end{align*}	
		It follows that
		\[\rk{f(x) - xA} = \rk{\Delta(x-1)} \le 1. \]
		If $r < p$, Equation~\eqref{eq:SubstitutingSymmetry}
		becomes
		\[f(x) =  apA + \sum_{j=1}^{r-1} \Delta(j).\]
		In particular $\im(f(x)-x\cdot A) \subseteq \sum_{i=1}^{p-2}{L_i}$.
		But by Theorem \ref{thm:APAP}, $\dim \sum_i L_i \leq 2$,  and hence $\rk{f(x) - xA} \leq 2$. 
	\end{proof}
	\begin{acks}
        TS, NT were partially supported by Swiss National Science Foundation (SNSF) project grant 200021 191981. TS was partially supported by Science Foundation -- Flanders (FWO) grant 1219723N. AV was supported by the Swiss National Science Foundation (SNSF)  grant 200142.
	\end{acks}

	\bibliographystyle{ACM-Reference-Format}
	\bibliography{bibliography}
	
\end{document}